\documentclass{article}
\usepackage{amsmath,amsthm,amssymb,hyperref,color}

\newtheorem{theorem}{Theorem}
\newtheorem{lemma}{Lemma}
\newtheorem{proposition}{Proposition}
\newtheorem{definition}{Definition}
\newtheorem{cor}{Corollary}

\newtheorem{remark}{Remark}

\newcommand{\norm}[1]{\ensuremath{\left\|#1\right\|}}
\newcommand{\abs}[1]{\ensuremath{\left\vert#1\right\vert}}
\newcommand{\ip}[2]{\ensuremath{\left\langle#1,#2\right\rangle}}

\newcommand{\Ber}{(\HH,X,G,k,\mu,d)}
\newcommand{\bbr}{\mathbb{R}}
\newcommand{\R}{\mathbb{R}}
\newcommand{\C}{\mathbb{C}}

\newcommand{\lip}{\langle}
\newcommand{\rip}{\rangle}
\newcommand{\ep}{\varepsilon}
\DeclareMathOperator{\supp}{supp}

\DeclareMathOperator{\spann}{span}

\newcommand{\calh}{\mathcal{H}}
\newcommand{\cala}{\mathcal{A}}

\newcommand{\BB}{\mathcal{B}}
\newcommand{\HH}{\mathcal{H}}

\newcommand{\UU}{\mathcal{U}}
\newcommand{\wto}{{\overset{\mbox{\tiny w}}{\to}}}
\renewcommand{\gg}{{\hat G \times G}}
\newcommand{\mm}{{\hat \mu \otimes \mu}}

\title{Invertibility of Positive Toeplitz Operators and Associated Uncertainty Principle}

\author{A. Walton Green and Mishko Mitkovski}

\begin{document}

\maketitle

\begin{abstract}
We study invertibility and compactness of positive Toeplitz operators associated to a continuous Parseval frame on a Hilbert space. As applications, we characterize compactness of affine and Weyl-Heisenberg localization operators as well as give uncertainty principles for the associated transforms.
\end{abstract}

\section{Introduction} The uncertainty principle in harmonic analysis is a fundamental principle roughly saying that a function and its Fourier transform cannot be simultaneously ``well localized''. There are many non-equivalent ways to make this statement precise by imposing different types of ``localization'' conditions. One such form of the uncertainty principle is the Benedicks' theorem~\cite{benedicks85} which says that a function $f\in L^2(\R^d)$ and its Fourier transform $\hat{f}$ cannot be both supported on sets of finite Lebesgue measure. The following is the quantitative form of this result as obtained by Amrein and Berthier~\cite{amrein77}: If $E, F\subseteq \R^d$ are sets of finite Lebesgue measure then there exists $c>0$ such that
\[ c\norm{f}_{L^2(\R)}\leq \norm{f}_{L^2(E^c)}+\|\hat{f}\|_{L^2(F^c)}, \] for all $f\in L^2(\R)$. Using the Kohn-Nirenberg quantization one can restate this result in the following form: Let $P$ be the operator \[ Pf(x) := \int_{\R^d}1_{(E\times F)^c}(x,\xi) \hat f(\xi) e^{2 \pi ix\cdot\xi}\, d\xi. \]
Then $\lip Pf,f \rip \ge c \|f\|^2$. 
It was shown by the first author \cite{green-thesis} that a similar result continues to hold for all subsets of $\R^{2d}$ with finite Lebesgue measure, including the ones which are not of the form $E\times F$. A natural analog of this result can be obtained by replacing the Kohn-Nirenberg quantization with the anti-Wick quantization. In this case we have the following statement for the short-time Fourier transform: Let $E \subset \R^{2d}$ have finite Lebesgue measure. There exists $c>0$ such that
\begin{equation}\label{eq:stft-intro} \int_{E^c} |V_\phi f(p,q)|^2 \, dp \, dq \ge c \|f\|^2 \end{equation}
for all $f \in L^2(\R^d)$, where $V_\phi f$ is the short-time Fourier transform defined by $V_\phi f(p,q) = \int f(x) e^{2\pi i px}\phi(x-q) \, dx$. This result was proved independently and almost simultaneously by Jaming~\cite{jaming98}, Janssen~\cite{janssen98}, and Wilczok~\cite{wilczok00} . Using similar reasoning Wilczok also showed the following analog for Wavelet transform: Let $E \subset (0,\infty) \times \bbr$ with finite \textit{affine} measure ($\int_E a^{-1} da db < \infty$). There exists $c>0$ such that
\begin{equation}\label{eq:wavelet-intro} \int_{E^c} |W_\psi f(a,b)|^2 \dfrac{da db}{a} \ge c \|f\|^2\end{equation}
for all $f \in L^2(\R^d)$ where the wavelet transform $W_\psi f$ is defined by $W_\psi f(a,b)=a^{-1/2}\int_{\bbr} f(x)\psi(a^{-1}x-b) \, dx$ for some wavelet $\psi$.

In the classical cases, when the window function $\phi$ or the mother wavelet function $\psi$ are specially chosen, $V_\phi L^2(\mathbb R^d)$ is a Fock space of analytic functions and $W_\psi L^2(\mathbb R^d)$ is a Bergman space on the upper half space. In both cases, these inequalities hold for the more general class of so called relatively dense sets, which is known to be the optimal such class of sets. This has been known since the 80s, and follows from the work of Luecking~\cite{luecking81,luecking85} on the Bergman space. For the Fock space, one can consult the work of Janson--Peetre--Rochberg~\cite{janson87}, Ortega-Cerd\`a~\cite{ortega98}, and the recent work of Ascensi~\cite{ascensi2015sampling}. However, it can be easily shown (see section \ref{sec:app} below as well as \cite{ascensi2015sampling,jaming21}) that relative density is a too weak condition for these results to continue to hold for general $L^2$ window and wavelet functions. 

Still, it is natural to ask whether there exists a larger class of sets $E$ (of infinite measure) for which (\ref{eq:stft-intro}) and (\ref{eq:wavelet-intro}) continue to hold for a more general class of window and wavelet functions. For the short-time Fourier transform some sufficient conditions can be found in \cite{ascensi2015sampling,jaming21} for windows with varying degrees of regularity. However, for general window functions $\phi \in L^2(\mathbb R^d)$, Fern\'andez and Galbis~\cite{fernandez10} showed that~(\ref{eq:stft-intro}) holds for all sets $E$ satisfying a certain thinness condition. The main goal of this paper is to show that~(\ref{eq:wavelet-intro}) also continues to hold for the appropriate analog of thin sets. Furthermore, we also extend the result of Fern\'andez and Galbis in the context of short-time Fourier transforms on more general LCA groups. We actually prove an uncertainty principle of Amrein-Berthier type for more general Berezin-type quantizations which include both short-time Fourier and wavelet inequalities as special cases. 

To state our result we now introduce the above mentioned Berezin-type quantization. Let $\HH$ be a Hilbert space and $(X, d, \mu)$ be a metric measure space with a metric $d$ and a Borel measure $\mu$. We assume that the metric $d$ is proper, i.e., every ball with respect to this metric is precompact. A collection of vectors $\{k_x\}_{x\in X}$ indexed by $X$ forms a normalized continuous Parseval frame for $\calh$ if
\begin{equation}\label{eq:parseval} f = \int_X \lip f,k_x \rip k_x \, d\mu(x) \end{equation}
holds for each $f \in \calh$. We will call the continuous map $k: X\to \HH$ given by $k(x)=k_x$ the Berezin quantization of $X$. We will assume that $X$ is a homogeneous space in the sense that some locally compact group $G$ acts transitively on $X$ in a way that both $\mu$ and $d$ are invariant under this group action ($\mu(gE)=\mu(E)$ and $d(gx,gy)=d(x,y)$). The group action needs also to respect the inner product on $\HH$ in the following way $|\lip k_{gx},k_{gy} \rip| = |\lip k_x,k_y \rip|$ for all $x,y \in X$, $g \in G$. Under all these assumptions, we will call the tuple $\Ber$ a \textit{Berezin quantization}.

Very often, the metric measure space $X$ and its quantization $k: X\to \HH$ of this type arise naturally whenever we are given a locally compact, second countable topological group $G$ and some of its nontrivial irreducible, square-integrable, unitary representations (if such exists), $\pi:G\to \UU(\calh)$. Then $G$ itself can be equipped with a  left-invariant Haar measure $\mu$, and left-invariant metric $d$ (inducing the topology on $G$) which is proper, and for any unit vector $k\in \HH$, the collection $\{k_x\}_{x\in G}$ defined by $k(x)=\pi(x) k$ will form a normalized continuous Parseval frame for $\calh$ such that $\lip k_{gx},k_{gy} \rip = \lip k_x,k_y \rip$ for all $x,y, g \in G$. Actually we can obtain a quantization with essentially the same properties using the same procedure even when we only have a projective irreducible square-integrable unitary representation, $\pi:G\to \UU(\calh)/\C$. 

The following inequality is the obvious analog of (\ref{eq:stft-intro}) and (\ref{eq:wavelet-intro}) for Berezin quantization. 

\begin{equation}\label{eq:ber-intro} \int_{X \backslash E} |\ip{f}{k_x}|^2 \, d\mu(x) \ge c \|f\|^2 \end{equation}
for all $f \in \calh$.
By choosing the group $G$ appropriately (see Section \ref{sec:app} for more details on this) one obtains both (\ref{eq:stft-intro}) and (\ref{eq:wavelet-intro}) as special cases.  

The inequality~(\ref{eq:ber-intro}) can be viewed as a statement about invertibility of a certain Toeplitz operator. For a non-negative, bounded, measurable function $\sigma: X\to \R$, the Toeplitz operator  $T_\sigma : \calh \to \calh$ with symbol $\sigma$ is defined by
\[ T_{\sigma}f = \int_X \sigma(x) \lip f,k_x \rip k_x \, d \mu(x). \] It is easy to see that each such operator is bounded, self-adjoint, and positive. We can easily recast~(\ref{eq:ber-intro}) in the following way:
\begin{equation}\label{eq:pos-intro}\lip T_{1_{X \backslash E}}f,f \rip \ge c \|f\|^2. \end{equation}
This is a statement about boundedness from below, and hence invertibility, of the Toeplitz operator $T_\sigma$ with an indicator symbol. It is clear that each Toeplitz operator whose symbol is bounded away from zero must be bounded from below (and hence invertible). However, to study (\ref{eq:pos-intro}), we must broaden this trivial result since an indicator function is only non-negative. So we would like to characterize the degree to which a non-negative symbol $\sigma:X\to \R$ can vanish and still generate an invertible Toeplitz operator $T_\sigma$.

The organization of the paper closely follows the main steps in the proof. The proof outline that we use goes back at least to Havin and Joricke~\cite{havin12}. Namely, to prove~(\ref{eq:pos-intro}) it suffices to show that the operator $T_{1_E}$ is compact with a trivial eigenspace. The later is proved by showing that from each nontrivial element of this eigenspace, using small translates, one can construct an infinite dimensional eigenspace of a slightly ``bigger'' compact self-adjoint operator of the same form. In Section~\ref{sec:compact} we address the compactness problem for positive Toeplitz operators. We show that compactness can be characterized in terms of the vanishing property of its Berezin transform, a condition which in turn is closely connected to the thinness condition of Fern\'andez and Galbis. In Section~\ref{sec:ind} we examine the linear independence problem for translations. We provide a fairly general condition (usually easy to check) on the group which guarantees linear independence of the translations. In Section~\ref{sec:thin} we combine the previous conclusions and state and prove our main result (Theorem~\ref{thm:compact2}). Finally, we show how our main result translates to more familiar settings, proving in particular~(\ref{eq:wavelet-intro}) for thin sets and a fairly wide class of mother wavelet functions.

\section{Compactness of positive Toeplitz operators}\label{sec:compact}

In this section we deal with the following question: Which non-negative symbols $\sigma:X\to\R$ generate compact Toeplitz operators $T_\sigma$? We show that, under suitable assumptions on the continuous Parseval frame $\{k_x\}_{x\in X}$, a necessary and sufficient condition for $T_\sigma$ to be compact is its ``diagonal" $\ip{T_\sigma k_x}{k_x}$ to vanish at infinity. We now make this precise. 

For a given bounded operator $T:\HH\to \HH$ we define its \textit{Berezin transform} $\tilde{T}:X\to \C$ by $\tilde{T}(x)=\ip{T k_x}{k_x}$. In the case when $T$ is a Toeplitz operator with symbol $\sigma$ we denote its Berezin transform by $\tilde \sigma$, i.e., $\tilde{\sigma}(x)=\ip{T_\sigma k_x}{k_x}$. In what follows, by $y \to \infty$, we will mean that $d(z,y) \to \infty$ for some (equivalently, all) $z \in X$.

\begin{definition} 
A measurable function $\sigma:X\to \mathbb{C}$ is said to be thin if its Berezin transform vanishes at infinity, i.e., 
	\begin{equation}\label{eq:compact} \tilde \sigma(y) :=  \int_X \sigma(x)|\lip k_x,k_y \rip|^2 \, d\mu(x) \to 0 \quad \mbox{as } y \to \infty
	\end{equation}
We will also say that a set $E \subset X$ is thin if the indicator function of $E$ satisfies (\ref{eq:compact}).
\end{definition}

The connection between compactness of a Toeplitz operator and the Berezin transform vanishing at infinity (called thinness here) is another quite old question \cite{cima1982carleson,mcdonald1979toeplitz,zhu1988positive}. The breakthrough results in this area were due to Axler and Zheng \cite{axler98compact} for the Bergman space on the disk, $A^p(\mathbb D)$, and to Su\'arez \cite{suarez} for $A^p(\mathbb B_n)$. Subsequently, these results been generalized to many different settings \cite{bauer2010heat,englis1999compact,smith2004reproducing}, even going beyond the realm of analytic function spaces \cite{batayneh16,bayer2015time,cordero2006symbolic,fernandez06,wong-book}. Since we are only dealing with nonnegative symbols, we are able to avoid many difficulties in these aforementioned results.

The following proposition gives a somewhat more explicit alternative characterization of thinness. 

\begin{proposition}\label{prop:eq}
Let $\sigma:X \to [0,1]$. The following are equivalent
	\begin{itemize}
	\item[(i)] $\sigma$ is thin.
	\item[(ii)] $\displaystyle\lim\limits_{y \to \infty} \int_{B(y,R)} \sigma(x) \, d\mu(x) = 0$ for some $R>0$.
	\item[(iii)] $\displaystyle\lim\limits_{y \to \infty} \int_{B(y,R)} \sigma(x) \, d\mu(x)  = 0$ for all $R>0$.
	\end{itemize}
\end{proposition}
\begin{proof}
Fix an element $e \in X$ which is arbitrary, but will be treated as the origin. Since $X$ is homogeneous, for each $x \in X$, there exists $g_x \in G$ (not necessarily unique) such that $g_xe=x$. For each $x$ we pick one such $g_x$ and fix it throughout the proof. By the invariance of $\mu$ and $d$,
  \[ \int_{B(y,R)} \sigma(x) \, d\mu(x) = \int_{B(e,R)}\sigma(g_yx) \, d\mu(x). \] 

Using our initial assumption, the metric $d$ is proper, i.e., every ball is precompact. So, the ball $B(e,R)$ can be covered by finitely many balls of a fixed radius. Therefore, we have (ii) implies (iii).

Next, consider the function $h(x) = \lip k_e,k_x \rip$. $h(e)=\|k_e\|^2 >0$ and $h$ is continuous so there exists $\delta >0$ such that $|h(x)| \ge h(e)/2$ for $x \in B(e,\delta)$. Thus,
	\[ \dfrac {h(e)^2}4 \int_{B(e,\delta)}\sigma(g_yx) \, d\mu(x) \le \int_{B(e,\delta)}\sigma(g_yx)|h(x)|^2 \, d\mu(x) \]
	\[\le \int_X \sigma(x) |h(g_y^{-1}x)|^2 \, d\mu(x) = \int_X \sigma(x)| \lip k_y,k_x \rip|^2 \, d\mu(x). \]
Therefore (i) implies (ii). To show (iii) implies (i), let $\ep >0$. We can find $R>0$ such that $\int_{B(e,R)^c} |h(x)|^2 \, d\mu(x) < \ep/2$. For this $R$, by (iii), there exists $N>0$ such that for $d(e,y) \ge N$, $\int_{B(e,R)}\sigma(g_yx) \, d\mu(x) \le \ep/(2\|k_e\|^4)$. Therefore,
	\[ \int_X \sigma(x)|\lip k_y,k_x \rip|^2 \, d\mu(y) = \int_{B(e,R)}+\int_{B(e,R)^c}\sigma(g_yx) |h(x)|^2 \, d\mu(x)  \le \ep,\] for all $y\in X$ with $d(e,y) \ge N$.
Here we used the fact that $|h(x)| \le \|k_e\| \cdot \|k_x\|$ but $\|k_x\|=\|k_e\|$. Indeed,
   \[ \|k_x\|^2 =|\lip k_x,k_x \rip| = |\lip k_e,k_e \rip| = \|k_e\|^2.\]
\end{proof}

We are now ready to characterize compactness of Toeplitz operators when the Parseval frame satisfies some additional decay conditions.

\begin{theorem}\label{thm:compact}
Let $\Ber$ be a Berezin quantization. Suppose that the continuous Parseval frame $\{k_x\}_{x\in X}$ satisfies the following: there exists a weight $w:X \to (0,\infty)$ and $M>0$ such that
\begin{itemize}
  \item[(i)] $\displaystyle w(y)^{-1}\int_X |\lip k_x,k_y\rip| w(x) \, d\mu(x) \le M$ for all $y \in X$.
  \item[(ii)] $\displaystyle\lim_{R \to \infty} \sup_{y \in X} w(y)^{-1} \int_{B(y,R)^c} |\lip k_x,k_y \rip| w(x) \, d\mu(x) =0$.
  \item[(iii)] $|\lip k_x,k_y \rip| \to 0$ as $d(x,y) \to \infty$.
  \end{itemize}
Then $\sigma \in L^\infty(X)$ is thin if and only if $T_\sigma$ is compact.
\end{theorem}

\begin{remark}\label{remark:schur}
  In the case $X=G$, it is often easier to check that $x \mapsto \lip k_1,k_x\rip w(x)$ is in $L^1(X)$ for some submultiplicative weight $w$ ($w(xy) \le C w(x)w(y)$). Then, (i) and (ii) follow from the group invariance of $|\lip k_x,k_y \rip|$. In this way, these conditions are related to the so-called analyzing vectors for the coorbit spaces of Gr\"ochenig and Feichtinger \cite{feichtinger1988unified}.
\end{remark}

\begin{proof} To prove necessity we only need to use (iii). First, we show that (iii) implies $k_y \wto 0$ as $y \to \infty$. We have $\lip k_y,k_x \rip \to 0$ for each $x \in X$ as $d(e,y) \to \infty$. Moreover, since $\{k_x\}_{x\in X}$ is a continuous Parseval frame we have $\norm{f}^2=\int_X\abs{\ip{f}{k_x}}^2d\mu(x)$, and hence $\overline \spann\{k_x\} = \calh$.
  Let $ f\in \HH$ and $\ep >0$. There exists $f_\ep \in \spann\{k_x\}$ such that $\|f-f_\ep\| \le \ep/(2\|k_e\|)$. Moreover, there exists $M$ such that if $d(e,y) > M$ then $|\lip k_y,f_\ep \rip| \le \ep /2$. Thus,
	\[ |\lip k_y,f \rip| \le \|k_e\| \cdot \|f-f_\ep\| + |\lip k_y,f_\ep \rip| \le \ep \]
whenever $d(e,y) > M$. Thus, if $T_\sigma$ is compact we have $T_\sigma k_y \to 0$ as $y\to\infty$, and hence $\tilde{\sigma}(y)=\ip{T_\sigma k_y}{k_y}\to 0$ as $y\to \infty$. 

We now prove sufficiency. For this we use (i) and (ii). 
Let $\ep>0$. There exists $R>0$ such that
\[ w(y)^{-1}\int_{B(y,R)^c} |\lip k_y,k_z \rip|w(z) \, d\mu(z) \le \ep \]
for all $y \in X$.
First we estimate the ``tails'' using the Schur property of $\{k_x\}$. For any $f \in \calh$, 
	\[ \int_X\left|\int_{B(x,R)^c} \sigma(y)\lip f, k_y \rip \lip k_y,k_x \rip \, d\mu(y) \right|^2 \, d\mu(x) \]
	\[ \le \|\sigma\|_\infty^2\int_X \int_{B(x,R)^c} |\lip f,k_y \rip|^2 |\lip k_y,k_x \rip|w(y)^{-1} \,d\mu(y) \int_{B(x,R)^c} |\lip k_y,k_x \rip|w(y) \, d\mu(y) \, d\mu(x)
	\]
        \[
        \le \ep \|\sigma\|_\infty^2\int_X |\lip f,k_y \rip|^2 \int_X w(x)|\lip k_y,k_x \rip|w(y)^{-1} \,d\mu(x) \, d\mu(y)
        \]
	\[
      	\le M \|\sigma\|_\infty^2 \ep \|f\|^2.\]

Now, since $\sigma$ is thin and bounded, by Proposition \ref{prop:eq}, there exists $S>0$ such that for $d(e,y) \ge S$, $\int_{B(y,R)}|\sigma|^2 \, d\mu \le \ep$. Thus,
	\[ \int_{B(e,S)^c}\left| \int_{B(x,R)} \sigma(y)\lip f,k_y \rip \lip k_y,k_x \rip \, d\mu(y) \right|^2 \, d\mu(x) \]
	\[
	 \le \ep \int_X\int_X |\lip f,k_y \rip|^2 |\lip k_x,k_y \rip|^2 \, d\mu(y)\, d\mu(x) = \ep \|k_e\|^2 \|f\|^2.\]
Now, let $f_n \wto 0$ be arbitrary. Then, $\|f_n\|$ is uniformly bounded.
	\begin{align*} 
		\|T_\sigma&f_n\|^2 =\int_X |\lip T_\sigma f_n,k_x \rip|^2 \, d\mu(x) \\
			&= \int_X\left|\int_{B(x,R)}+\int_{B(x,R)^c} \sigma(y)\lip f_n,k_y \rip \lip k_y,k_x \rip \, d\mu(y)\right|^2 \, d\mu(x) \\
			&\le \int_{B(e,S)} + \int_{B(e,S)^c} 2\left|\int_{B(x,R)} \sigma(y)\lip f_n,k_y \rip \lip k_y,k_x\rip\, d\mu(y)\right|^2\, d\mu(x) + C \ep \\
			&\le 2\int_{B(e,S)}\left|\int_{B(x,r)} \sigma(y)\lip f_n,k_y \rip \lip k_y,k_x \rip \, d\mu(y)\right|^2\, d\mu(x) + C \ep.
	\end{align*}

Since the integrand goes to zero pointwise, applying the dominated convergence theorem we obtain
	\[ \limsup_{n \to \infty} \|T_\sigma f_n\|^2 \le C \ep. \]
But $\ep$ is arbitrary so $\lim\limits_{n\to\infty} T_\sigma f_n =0$. Therefore $T_\sigma$ is compact.

\end{proof}

\section{Independence of Translations} \label{sec:ind}

The action of $G$ on $X$ naturally induces the following translation action on $L^2(X, d\mu)$.
For a given $h \in G$, the translation operator $\tau_h: L^2(X, d\mu)\to L^2(X, d\mu)$ is defined by $\tau_h F(x)=F(h^{-1} x)$ for $F \in L^2(X, d\mu)$. We denote the translated function by $F_h:=\tau_h F$. We can view $\calh$ as a subspace of $L^2(X)$ by identifying $f \in \HH$ with $x \mapsto\lip f,k_x \rip \in L^2(X, d\mu)$. In this way, $\tau_h$ induces the following shift operator $S_h:\calh \to \calh$ defined by
\[ S_h f = \int_X \lip f,k_{h^{-1}x} \rip k_{x} \, d\mu(x) =  \int_X \lip f,k_{x} \rip k_{hx} \, d\mu(x). \]
It is easy to check that $S_h^*=S_{h^{-1}}$ and $S_h$ is bounded. While it is not true in general that $S_hk_x=k_{hx}$, this does hold if $\lip k_{hx},k_{hy} \rip= \lip k_x,k_y \rip$, and then $S_h$ is unitary so that $S_hT_\sigma S_h^* = T_{\sigma_h}$. For this reason, we introduce the subgroup of $G$,
	\begin{equation}\label{eq:gtau}G_\tau =\{ h \in G : \lip k_{hx},k_{hy} \rip = \lip k_x,k_y \rip \mbox{ for all }x,y \in X \}. \end{equation}
In this way, $h \mapsto S_h$ is a homomorphism from $G_\tau \to \mathcal{U}(\calh)$.

We are interested in conditions on the action under which different shifts of any given non-zero $f\in \HH$ form a linearly independent set. Observe first that this problem can be reduced to the more studied problem of linear independence of translations. Indeed, for any nonzero $f\in \HH$ define $F(x) = \lip f,k_x \rip$. Then $F \in L^2(X)$ and $ \lip S_h f,k_x \rip = F(h^{-1}x)$ for $h \in G_\tau$. The last relation clearly implies that the linear independence of $\{S_{h_1}f, S_{h_2}f, \dots, S_{h_n}f\}$ is equivalent to the linear independence of the corresponding translations of $F$, $\{F_{h_1}, F_{h_2}, \dots, F_{h_n}\}$. 

Note, however, that the linear independence problem for translations is in general a very difficult problem. Even for some fairly simple group actions, translations may be not linear independent. Specifically, in the case of the Affine group, $\chi_{[0,1]}(x)=\chi_{[0,1]}(2x)+\chi_{[0,1]}(2x-1)$ is a simple example. Moreover, for the Heisenberg Group, the linear independence problem for time-frequency shifts (known as the HRT conjecture \cite{heil06}) is still open and is widely considered to be very difficult.

However, in the case of Abelian groups $G$, a theorem of Edgar and Rosenblatt~\cite{edgar79} says that the translations of any non-zero $F\in L^2(G)$ are linearly independent as long as $G$ has no nontrivial compact subgroups. Our main goal will be to extend this result to the case of homogeneous spaces $X$. 

Let $\Gamma \subset G$. Denote 
\[ \mathbb{C}\Gamma = \{ \sum c_nh_n : h_n \in \Gamma, c_n \in \mathbb{C}\}. \]
The elements $\theta \in \mathbb{C}\Gamma$ act on $F \in L^2(X)$ by
\[ \theta F = \sum c_n F_{h_n}. \]
Thus, a collection of translations $\Gamma F = \{ F_h : h\in \Gamma\}$ is linearly independent if and only if $\theta F \ne 0$ for all $0 \ne \theta \in \mathbb{C}\Gamma$. This leads to the following definition.

\begin{definition}
  We say that $L^2(X)$ has linearly independent $\Gamma$-translations if for all $F \in L^2(X)$, $\theta \in \mathbb{C}\Gamma$,
  \[ \theta F =0 \mbox{ implies }F=0 \mbox{ or } \theta =0. \]
\end{definition}

We want to study the independence of $\Gamma$-translates of $L^2(X)$ by reducing it to the case $L^2(\Gamma)$ where we can use the result of Edgar and Rosenblatt.
The following proposition establishes this reduction. It can be viewed as a generalization of a result from \cite{linnell17} to non-discrete $\Gamma$.

\begin{proposition}\label{prop:trans}
Let $(X,\mu)$ be a measure space. Let $\Gamma$ be a unimodular locally compact group $\Gamma$ acting on $X$ under which $\mu$ is invariant.
 Then $L^2(X)$ has linearly independent $\Gamma$-translations if $L^2(\Gamma)$ has linearly independent $\Gamma$-translations.
\end{proposition}

In proving this, we will make crucial use of the following identity.
\begin{lemma}\label{lemma:poisson}
Let $\Gamma$ and $(X,\mu)$ be as above. Then there exists a measure $\nu$ on $X / \Gamma$ such that
  \begin{equation}\label{eq:poisson} \int_{X / \Gamma} \int_\Gamma F(\gamma x) \, d\lambda(\gamma) \, d\nu(\Gamma x) = \int_X F(x) \, d\mu(x) \end{equation}
	holds for all $0 \le F \in L^1(X,\mu)$, where $\lambda$ is the Haar measure on $\Gamma$.
\end{lemma} We are not aware of such a fact in the literature. However, in the case where $X$ is a group and $\Gamma \le X$, it is well-known \cite{folland-book,mackey52}, so we show at the end of this section that a slight modification of these proofs yields (\ref{eq:poisson}) in our setting.

\begin{proof}[Proof of Proposition \ref{prop:trans}]
  Let $F \in L^2(X)$ be nonzero. Let $E$ denote the set of $\Gamma x \in X /\Gamma$ with $\int |F(\gamma x)|^2\,d \lambda(\gamma) \in \{0,\infty\}$. By formula (\ref{eq:poisson}), replacing $F$ with $|F|^2$, it must be that $\nu(E^c)>0$. So, there exists $\Gamma x \in E^c$ such that $g(\gamma) := F(\gamma x)$ satisfies $g \in L^2(\Gamma)$ and $g \ne 0$. Thus, for each $\theta \in \C\Gamma$, if $\theta F=0$, then $\theta g=0$ which implies $\theta =0$.
\end{proof}

Applying the known result for translations on Abelian groups from \cite{edgar79}, we obtain the following corollary.

\begin{cor}\label{cor:ind}
  Let $\Gamma$ be an Abelian subgroup of $G_\tau$ with no nontrivial compact subgroups. For any $\{h_n\} \subset \Gamma$, and $f \in \calh$, $\{S_{h_n}S_{h_{n-1}}\cdots S_{h_1}f\}_{n=1}^\infty$ is linearly independent.
\end{cor}

\begin{proof}
Let $f \in \calh$. Define $F(x) = \lip f,k_x \rip$. $F \in L^2(X)$ and $ \lip S_h f,k_x \rip = F(h^{-1}x)$. Therefore since $L^2(\Gamma)$ has linearly independent $\Gamma$ translations by the result of Edgar and Rosenblatt \cite{edgar79}, $L^2(X)$ has linearly independent $\Gamma$ translations by Proposition \ref{prop:trans}.
\end{proof}

  Finally, we give the proof of Lemma \ref{lemma:poisson}.
  Let $\lambda$ be the Haar measure for $\Gamma$. Since $\Gamma$ is unimodular, the Haar measure is both right and left invariant and $d\lambda(\gamma) = d\lambda(\gamma^{-1})$. Define $P:C_c(X) \to C_c(X / \Gamma)$ by
  \[ Pf(\Gamma x) = \int_\Gamma f(\gamma x) \, d\lambda(\gamma). \]
  This is well-defined since if $\Gamma y=\Gamma x$, then $\gamma_0 y=x$ so $Pf(\Gamma y)=Pf(\Gamma x)$ by the invariance of $\lambda$.

  We claim that $Pf=0$ implies $\int_Xf \, d\mu =0$. Indeed, take $\phi \in C_c(X)$ such that $P\phi = 1$ on $\Gamma \supp f$ (this can be done by surjectivity of $P$ which we will prove below). Then,
  \[ 0 = \int_X \phi(x) \int_\Gamma f(\gamma x) \, d\lambda(\gamma) \, d\mu(x) = \int_X \int_\Gamma \phi(\gamma^{-1}x)\, d\lambda(\gamma) \, f(x) \, d\mu(x) \]
  \[= \int_X P\phi(\Gamma x)f(x) \, d\mu(x) = \int_X f(x) \, d\mu(x).\]
  Now we are ready to define the measure $\nu$. Define the functional on $C_c(X / \Gamma)$ by $Pf \to \int_X f \, d\mu$. This is a well-defined positive linear functional by the previous claim and Lemma \ref{lemma:p} below. By the Riesz Representation Theorem, there exists $\nu$ such that
  \[ \int_{X / \Gamma} Pf \, d\nu = \int_X f \, d\mu \]
  which is the desired identity (\ref{eq:poisson}). It only remains to prove the surjectivity of $P$.

  \begin{lemma}\label{lemma:p}
    Let $\phi \in C_c(X/\Gamma)$. There exists $f \in C_c(X)$ such that $Pf=\phi$, $\Gamma (\supp f) = \supp \phi$, and $f \ge 0$ if $\phi \ge 0$.
  \end{lemma}
  \begin{proof}
    Denote by $E=\supp \phi$. Since $E$ is compact, there exists $\{g_k\}_{k=1}^N \subset G$ such that $E \subset \cup_{k=1}^N \Gamma g_kV$ where $V$ is a compact neighborhood of $e$. $V$ has an open neighborhood $U$ with compact closure. Then, consider the function $\psi(x)= \max_k \{ d(g_k^{-1}x,U^c)\}$. $\psi \in C_c(X)$, $\psi > 0$ on $\cup_{k=1}^N g_kV$, and for $\Gamma x \in E$, 
    \[ P\psi(\Gamma x) = \int_\Gamma \psi(\gamma x) \, d\lambda(\gamma) > 0. \]
    Now, setting
    \[ f(x) = \dfrac{\phi(\Gamma x)}{P \psi (\Gamma x)} \psi(x), \]
    we see that $f \in C_c(X)$, $\Gamma \supp f = \supp \phi$, and $f \ge 0$ if $\phi \ge 0$. We only need to check
    \[ Pf(\Gamma x) = \int_\Gamma \dfrac{\phi(\Gamma \gamma x)}{P \psi (\Gamma \gamma x)}\psi(\gamma x) \, d\lambda(\gamma) \]
    \[ =\dfrac{\phi(\Gamma x)}{P \psi (\Gamma x)})P\psi(\Gamma x) = \phi(\Gamma x).\]
  \end{proof}

\section{Thin Toeplitz Operators}\label{sec:thin}
In this section, we show that from one compact Toeplitz operator $T_\sigma$, one can create $T_\rho$ with $\rho \ge \sup_{n}\sigma_{h_n}$ for some infinite, but small, sequence of translates $\{h_n\}$, while $T_\rho$ remains compact. 

The set of positive linear operators $B(\calh)^+$ on a Hilbert space $\calh$ possesses a partial ordering. We say $A \ge B$ if $A-B$ is a positive operator, i.e  $\lip (A-B)f,f\rip \ge 0$ for all $f \in \calh$. Letting $(\sigma \vee \rho)(x) = \max\{\sigma(x),\rho(x)\}$, we have
\begin{itemize}
\item[(i)] $T_\sigma,T_\rho \le T_{\sigma \vee \rho}$.
\item[(ii)] If $T_\sigma$ and $T_\rho$ are compact, so is $T_{\sigma \vee \rho}$.
\item[(iii)] $\|T_{\sigma \vee \rho}\| \le \max\{ \|\sigma\|_\infty, \|\rho\|_\infty \}$
\end{itemize}
Since $\lip T_\sigma f,f \rip = \int \sigma(x) |\lip f,k_x\rip|^2 \, d\mu(x)$, ordering of the Toeplitz operators follows from the ordering of their symbols. (ii) follows from the fact that $T_{\sigma\vee \rho} \le T_\sigma + T_\rho$. (iii) is a consequence of the trivial estimate that $\|T_\sigma\| \le \|\sigma\|_\infty$.

\begin{lemma}\label{lemma:2fun}
Let $\sigma$ be thin and $\{h_k\} \subset G$ such that $h_k \to 1$. Then, there exists $\{m_k\}$ such that
    \[ \rho(x) = \sup_k \sigma_{h_{m_k}h_{m_{k-1}}\cdots h_{m_1}}(x) \]
is thin.
\end{lemma}

\begin{proof}
  Let $K_y(x)= |\lip k_x,k_y \rip|^2 \in L^1(X)$.
  Then,
  \[ \tilde \sigma(y) = \int_X \sigma(x)K_y(x) \, d\mu(x). \]
  Therefore, by the boundedness of $\sigma$ and the continuity of the translations on $L^1(X)$,
    \[ |\tilde \sigma(y)-\tilde \sigma_h(y)| = \int_X \sigma(x)[ K_y(x)-K_y(hx)] \, d\mu(x) \to 0 \]
    as $h \to 1$ for each $y$. 
    Since $h_m \to 1$ we can pick a subsequence such that $\{\prod_{finite}h_m\} \subset B_1$. From this subsequence, we pick another one. Set $\rho_0=\sigma$. Then, since $\sigma$ is thin, there exists $R_0$ such that $\tilde \rho_0 \le 1$ outside $B_{R_0}$. Then, pick $m_0$ such that
    \[ |\tilde \rho_0 - \tilde \rho_{0,h_{m_0}}| \le 1 \]
    on $B_{R_0+1}$. Then,
    \[ |\tilde \rho_0 - \tilde \rho_{0,h_{m_0}}| \le 3.\]
    Now, having $\rho_j$ (thin) and $h_{m_j}$ for $j < k$, pick $m_k$ in the same manner as above (finding $R_k$ since $\rho_j$ are all thin) so that
    \[ |\tilde \rho_j  - \tilde \rho_{j,h_{m_k}}| \le 3\cdot 2^{-k} \]
    for all $j < k$. then set $\rho_k = \max\{\rho_{k-1},\rho_{k-1,h_{m_k}}\}$.

    Now we are ready to prove that $\rho=\lim \rho_k$ is thin. Let $\ep >0$ and pick $j$ such that $3\sum_{k=j}^\infty 2^{-k} \le \tfrac \ep 2$. Let $R>0$ such that $\tilde \sigma(y) \le \tfrac \ep{2^{j+1}}$ for $d(1,y) \ge R$. Then, for $d(1,y) \ge R+1$, $\tilde \sigma_{h}(y) \le \tfrac \ep{2^{j+1}}$ for all $h \in B_1$. Then, by the fact that $|P_j|=2^j$ ($P_j$ is the power set of $\{0,1,2,\ldots,j\}$), and noticing that $\rho_k = \max_{I \in P_k} \sigma_{\prod_{i \in I}h_{m_i}}$,
	\[ \tilde \rho_j(y) \le \sum_{I \in P_j} \tilde \sigma_{\prod_{i \in I}h_{m_i}}(y) \le 2^j(\tfrac \ep {2^{j+1}}) = \tfrac \ep 2. \]
By construction of $h_{m_k}$, $|\tilde \rho_k -\tilde \rho_{k+1}| \le 2^{-k}$. Therefore,
	\[ \tilde \rho(y) \le \tilde \rho_j(y) + \sum_{k=j}^\infty |\tilde \rho_k(y)-\tilde \rho_{k+1}(y)| \le \tfrac \ep 2 + 3\sum_{k=j}^\infty 2^{-k} \le \ep.\]

\end{proof}

\subsection{Main Results}

We now assemble the pieces from the previous sections, specifically Theorem \ref{thm:compact} and Corollary \ref{cor:ind} to prove our main results.

\begin{theorem}\label{thm:compact2}
  Let $\Ber$ be a Berezin quantization. Suppose there exists $w:X \to (0,\infty)$ and $M>0$ such that
  \begin{itemize}
  \item[(i)] $\displaystyle w(y)^{-1}\int_X |\lip k_x,k_y\rip| w(x) \, d\mu(x) \le M$ for all $y \in X$.
  \item[(ii)] $\displaystyle\lim_{R \to \infty} \sup_{y \in X} w(y)^{-1} \int_{B(y,R)^c} |\lip k_x,k_y \rip| w(x) \, d\mu(x) =0$.
  \end{itemize}
  Suppose also that $G_\tau$ contains an Abelian subgroup which has an accumulation point at $1$ and no nontrivial compact subgroups. Then, if $\sigma:X \to [0,1]$ is thin, there exists $c>0$ such that
  \[ \lip T_{1-\sigma}f,f \rip \ge c \|f\|^2 \]
  for all $f \in \calh$.
\end{theorem}

\begin{proof}
Suppose there exists $f \ne 0$ such that $T_\sigma f=f$. By Lemma \ref{lemma:2fun}, there exists a sequence $\{h_k\}_{k=1}^\infty \subset \Gamma$ such that $\rho(x)=\sup_{k} \sigma_{h_k \cdots h_1}(x)$ is thin. Define $f_k = S_{h_{k}}\cdots S_{h_1}f$. Then,
    \[ 1 \ge \lip T_\rho f_k,f_k \rip \ge \lip S_{h_k}\cdots S_{h_1}T_\sigma(S_{h_k}\cdots S_{h_1})^*f_k,f_k \rip = \lip T_\sigma f,f \rip = 1.\]
    This implies $T_\rho f_k=f_k$.
    However, the collection $\{f_k\}$ is linearly independent by Corollary \ref{cor:ind} since $\Gamma$ is Abelian, has no nontrivial compact subgroups, and is contained in $G_\tau$.
Therefore the dimension of the eigenspace corresponding to the eigenvalue $1$ is infinite. However, $T_\rho$ is compact by Theorem \ref{thm:compact} which is a contradiction.
\end{proof}

Properties (i) and (ii) often follow from the integrability of $\lip k_x,k_e \rip$, see Remark \ref{remark:schur}. In case the decay of $\lip k_x,k_y\rip$ is not sufficient for (i) and (ii), we can obtain something weaker.

\begin{theorem}\label{thm:compact3}
  Let $\Ber$ be a Berezin quantiztion and suppose $G_\tau$ contains an Abelian subgroup which has an accumulation point at $1$ and no nontrivial compact subgroups. If $\sigma:X \to [0,1]$ is in $L^1(X)$, then there exists $c>0$ such that
  \[ \lip T_{1-\sigma}f,f \rip \ge c \|f\|^2 \]
  for all $f \in \calh$.
\end{theorem}

\begin{proof}
First, let us see that $\sigma \in L^1(X)$ implies $T$ is compact. Indeed, if $f_n \wto 0$, then $\lip f_n,k_x \rip \to 0$ for each $x \in X$. Then, by dominated convergence,
\[ \lip T_\sigma f_n,f_n \rip = \int \sigma(x) |\lip f_n,k_x\rip|^2 \, d\mu(x) \to 0. \]
Secondly, since the translations are continuous on $L^1(X)$, one can find $h_k$ such that $\|\sigma_{h_k}-\sigma\|_{L^1} \le 2^{-k}$ which implies $\rho(x) = \sup_{k}\sigma_{h_k\cdots h_1}(x)$ is still in $L^1$. Then we follow the same proof as Theorem \ref{thm:compact2}.
\end{proof}

\section{Applications}\label{sec:app}

\subsection{Short-Time Fourier Transform on LCA groups}
Theorem 3 is similar to Benedicks Theorem (also called the Amrein-Berthier Theorem or Qualitative Uncertainty Principle) for the Plancherel groups, proved in \cite{arnal97}, which states that if $f \in L^2(G, \mu)$ and $\hat f \in L^2(\hat G, \hat \mu)$ are each supported on sets of finite $\mu$ and $\hat \mu$ measure respectively, then $f=0$. This uncertainty principle limits the joint time-frequency distribution of these functions. Namely, it states that the joint time-frequency support (a set in $\hat G \times G$) cannot be contained in a set of finite $\hat \mu \otimes \mu$ measure.

However, looking at other joint time-frequency distributions, such as the Wigner distribution, Ambiguity function, or Short-Time Fourier transform (STFT) yields a stronger result. We will only focus on the STFT, which is defined, for $f,\phi \in L^2(G)$,
\[ V_\phi f(p,q) = \int_{G} f(t)\overline{p(q^{-1}t)\phi(q^{-1}t)} \, d\mu(t) = p(q)\hat(f\phi_q)(p)\]
for $(p,q) \in \hat G \times G$. We will only deal with locally compact groups $G$ which are Abelian and second countable. The second countability is used to ensure that the invariant metric $d$ is proper. In this way, by the Plancherel theorem for $G$, we obtain
\[ \lip V_\phi f, V_\psi g \rip = \int_G \lip \hat (f \cdot \overline{\phi_q}),\hat (g \cdot \overline{\psi_q})\rip_{L^2(\hat G,\hat \mu)} \, d\mu(q) \]
\[ = \int_G \int_G f(x)\overline{\phi(q^{-1}x)g(x)}\psi(q^{-1}x) \, d\mu(x) \, d\mu(q) = \lip f,g \rip \cdot \lip \psi,\phi \rip.\]

Which shows that $\phi_{p,q}(x) = p(q^{-1}x)\phi(q^{-1}x)$ is a Parseval frame for $L^2(G)$ if $\|\phi\|=1$. This implies $\|\phi\| \cdot \|f\| = \|V_\phi f\|$. In order to apply Theorem \ref{thm:compact2}, we take $\gg$ to be the homogeneous space with the measure $\hat \mu \otimes \mu$ acted on by the ``Heisenberg'' group $H(G):= \gg \times \mathbb{T}$ with the operation
\[ (p',q',z')(p,q,z) := (pp',qq',zz' p(q')).\]
$H(G)$ acts on $\gg$ by $(p,q,z)(p',q') = (pp',qq')$. It can also be checked that
\[ |\lip \phi_{(p,q,z)(p',q')},\phi_{(p,q,z)(p'',q'')} \rip| = |\lip \phi_{p',q'},\phi_{p'',q''} \rip|\]
and moreover $\lip \phi_{(1,q,1)(p',q')},\phi_{(1,q,1)(p'',q'')} \rip = \lip \phi_{p',q'},\phi_{p'',q''} \rip$ so that $\{1_{\hat G}\} \times G \times \{1\} \le H(G)_\tau$. If we had used the more conventional definition of $V_\phi f(p,q)$ to be $\hat (f \phi_q)(p)$ then $\hat G \times \{1_G\} \times \{1\} \le H(G)_\tau$ instead.

\begin{theorem}\label{thm:compact-stft}
  Let $G$ be a noncompact second countable LCA group. For any $\phi \in L^2(G)$, $\sigma:\gg \to [0,1]$ is thin if and only if
  \[ T_\sigma f = \int_\gg \sigma(p,q)\lip f,\phi_{p,q} \rip \phi_{p,q} \, d(\mm)(p,q)\]
  is compact.
\end{theorem}
\begin{proof}
	To apply our compactness characterization, Theorem \ref{thm:compact}, let us first show that for $\phi \in L^2(G)$, $\lip \phi_{p,q},\phi_{p',q'} \rip \to 0$ as $d((p,q),(p',q')) \to \infty$. By invariance of the frame, it enough to check $\lip \phi,\phi_{p,q} \rip \to 0$ as $(p,q) \to \infty$. First, if $p \to \infty$ and $q$ remains bounded, then for each $q$, by the Riemann-Lebesgue Lemma $\lim_{p \to \infty}\lip \phi,\phi_{p,q}\rip = 0$. Therefore since $q$ remains in a compact set, $\lip\phi,\phi_{p,q} \rip \to 0$. Otherwise, we consider $q \to \infty$. Let $\ep>0$. Since $\phi \in L^2(G)$, there exists $R$ such that
  \[ \int_{B(1_G,R)^c} |\phi|^2 \, d\mu \le \ep. \]Then, there exists $M$ such that for $d(q,1) > M$, $B(q^{-1},R) \subset B(1_G,R)^c$. Then,
  \[ |\lip \phi,\phi_{(p,q)} \rip| \le \int_{B(1_G,R)} + \int_{B(1_G,R)^c} |\phi| |\phi(q^{-1}\cdot)| \, d\mu  \le 2\|\phi\| \ep. \]
  Applying Theorem \ref{thm:compact} gives the necessity. 
  To show sufficiency, 
consider the Feichtinger algebra
\[S_0(G):=\{ \phi \in L^2(G) : V_\phi\phi \in L^1(\gg) \}.\]
It is known that $S_0(G)$ is dense in $L^2(G)$, see \cite[Lemma 4.19]{jakobsen18}.
First note that $T_\sigma$ is compact for any $\phi \in S_0(G)$ by Theorem \ref{thm:compact}. Then,
  \[ \lip T_\sigma^{\phi,\psi}f,g \rip = \int \sigma(p,q)\lip f,\phi_{p,q} \rip \lip \psi_{p,q},g \rip \]
  \[\le \|\sigma\|_\infty \|f\|_{L^2(G)} \|g\|_{L^2(G)} \|\psi\|_{L^2(G)} \|\phi\|_{L^2(G)}.\]
  Therefore, we have $\|T_\sigma^{\phi,\psi}\| \le \|\sigma\|_\infty \|\phi\|_{L^2(G)} \|\psi\|_{L^2(G)}$. Then, $T_\sigma^{\phi,\phi}-T_\sigma^{\phi',\phi'} = T_\sigma^{\phi-\phi',\phi} + T_\sigma^{\phi,\phi-\phi'}$.
  This concludes the proof since compact operators are closed in the operator norm topology. 
\end{proof}

Putting this together with Theorem \ref{thm:compact2}, we obtain the following uncertainty principle (see \cite{fernandez10} for the case $G=\bbr^d$ and $\sigma$ an indicator function).

\begin{theorem}\label{thm:uncp-stft}
  Let $G$ be a second countable LCA group containing a subgroup $\Gamma$ such that $\Gamma$ has an accumulation point at $1$ and $\Gamma$ has no nontrivial compact subgroups. Let $\|\phi\|_{L^2(G)}=1$. If $\sigma:\gg \to [0,1]$ is thin, then there exists $c>0$ such that
  \[ \lip T_{1-\sigma}f,f \rip = \int_{\gg} (1-\sigma)|V_\phi f|^2 \, d(\mm) \ge c \|f\|^2 \]
  for all $f \in L^2(G)$.
\end{theorem}

\begin{proof}
The subgroup $\{1_{\hat G}\} \times \Gamma \times \{1\} \le H(G)$ gives the independence of translations and the tautology that $\sigma$ is thin if and only if $T_\sigma$ is compact concludes the proof using the argument of Theorem \ref{thm:compact2}.
\end{proof}

\begin{cor}
Under the assumptions of Theorem \ref{thm:uncp-stft}, if $E \subset \gg$ is thin, then there exists $c>0$ such that
	\[ \int_{\gg \backslash E} |V_\phi f|^2 \, d\mm \ge c \|f\|^2 \]
for all $f \in L^2(G)$.
\end{cor}

In particular, $\supp V_\phi f$ cannot be thin unless $f=0$.

\begin{remark}
If $\phi$ and $f$ are both supported on compact sets $K_1$ and $K_2$, then $\supp V_\phi f \subset \hat G \times K_2K_1^{-1}$. This shows the thinness condition cannot be relaxed from vanishing at infinity to some smallness at infinity condition.
\end{remark}

\subsection{Wavelet Transform}
Next, we apply these results to wavelets, which are a special case of the affine group acting on $L^2(\bbr^d)$. For $\phi \in L^2(\bbr^d)$, define the following unitary representation of the wavelet group $G = (0,\infty) \times \bbr^d$ on $L^2(\bbr^d)$ by
\begin{equation}\label{eq:action} \phi_{a,b}(x) = a^{-1/2}\phi(a^{-1}x-b) \end{equation}
for $a \in (0,\infty)$ and $b \in \bbr^d$. The group operation is then $(a_1,b_1)(a_2,b_2) = (a_1a_2,a_2^{-1}b_1+b_2)$ and the Haar measure is $d\lambda(a,b) = a^{-1} \, da \, db$.

We say that a wavelet $\phi \in L^2(\bbr^d)$ is admissible if
\begin{equation}\label{eq:adm} \int_0^\infty \dfrac{|\hat \phi(x \xi)|^2}{\xi} \, d\xi =1 \end{equation}
for a.e. $x \in \bbr^d$. Due to the invariance of the measure $a^{-1} \, da$, it is enough to check this for almost every $x \in \mathbb{S}^{d-1}$. We define the set $\cala$ to be the set of all admissible $\phi$.

If $\phi$ is admissible, then $\{\phi_g\}_{g \in G}$ does form a generalized Parseval frame with the measure $\lambda$.
To see this, define the wavelet transform $W_\phi f(a,b) :=\lip f,\phi_{a,b} \rip$. Then, since $\int e^{-ib \xi}W_\phi f(a,b) \, db = a^{-1/2}\hat f(a^{-1}\xi) \hat \phi (\xi)$, by the Plancherel theorem on $\bbr^d$,
\[ \int_{G} |W_\phi f|^2 \, d\lambda = \int_0^\infty \int_{\bbr^d} | \hat f(a^{-1}\xi) \hat\phi(\xi)|^2 a^{-1}\, d\xi a^{-1}\, da \]
\[= \int_{\bbr^d} |\hat f(\eta)|^2 \int_0^\infty |\hat \phi(a \eta)|^2 \, a^{-1} da \, d\eta = \|f\|^2.\]

We want to check the Schur condition (i) and (ii) in Theorem \ref{thm:compact2} so we study the decay properties of $W_\phi \phi$. Define
\[ \BB^1_w = \{ \phi \in L^2(\R^d) : W_\phi\phi \in L^1(w\, d\lambda) \} \]
for a weight $w$. We can show that a very large class of functions is contained in $\BB^1_w$. Define the translation operator $\tau_hf(x) = f(x+h)$. For $0<\alpha \le 1$, denote by $\Lambda_\alpha$ the class of $L^1$ functions such that $\|\tau_h f-f\|_{L^1} \le Ch^\alpha$. $\Lambda_1$ contains the Schwarz functions as well as less smooth functions like indicator functions (thus including the Haar wavelet).
\begin{lemma}
  Let $0 < \ep < \alpha \le 1$ and $w_\ep(a,b) = a^{d/2+\ep}$. Then,
  \[ \Lambda_\alpha \cap L_0^1(|x|^\alpha) \subset \BB^1_{w_\ep} \]
	where $L^1_0(|x|^\alpha) = \{ f \in L^1(\R^d) : \int f =0 \mbox{ and } \int |f(x)| \, |x|^\alpha \, dx < \infty \}$.
\end{lemma}

In particular, this weight is multiplicative, so by Remark \ref{remark:schur}, $\phi_{a,b}$ satisfies the Schur conditions (i) and (ii) in Theorem \ref{thm:compact2} for any $\phi \in \Lambda_\alpha \cap L^1_0(|x|^\alpha)$.

\begin{proof}
  We split $(0,\infty)=(0,1) \cup (1,\infty)$. On $(0,1)$,
  \[ \int_{\bbr^d} \int_0^1 \left|\int \phi(x) \phi(a^{-1}x-b) \, dx \right|a^{-d/2}  w(a) \, \dfrac{da}{a} \, db  \]
  \[\le \|\phi\|_{L^1}^2 \int_0^1 a^{d/2+\ep}\, \dfrac{da}{a^{d/2+1}} < \infty \]
  On the other hand, using the mean zero property of $\phi$, we estimate
  \[ \int |W_\phi\phi(a,b)| \, db = a^{-d/2}\int \left|\int \phi(x) \left[ \phi(a^{-1}x-b) - \phi(-b) \right] \, dx\right| \, db \]
  \[\le a^{-d/2}\int |\phi(x)| \|\tau_{a^{-1}x}\phi-\phi\|_{L^1} \, dx. \le C a^{-d/2-\alpha} \int |\phi(x)|\cdot |x|^\alpha \, dx\]
  Therefore,
  \[ \int_{\bbr^d} \int_1^\infty |W_\phi\phi| w \, d\lambda \le C\int_1^\infty \dfrac{w(a)}{a^{d/2+1+\alpha}} \, da \|\phi|x|^\alpha\|_{L^1}.\]
  Taking $\ep<\alpha$ ensures that the $a$ integral is finite.
\end{proof}

This lemma gives us plenty of information about the space $\BB^1_w$. It can be verified that the admissibility condition (\ref{eq:adm}) holds for any radial, normalized mean zero function in $L^1 \cap L^1(|x|) \cap L^2$. From this discussion and the previous lemma, we have
\[ L^1 \cap L^1_0(|x|) \cap L^2 \cap (\cup_{0<\alpha \le 1}\Lambda_\alpha) \cap \{ \phi \mbox{ radial }\} \]
\[\subset (\cup_{0<\ep<1}\BB^1_{w_\ep}) \cap \cala=:\cala_1. \]
Therefore, by Theorem \ref{thm:compact}, the following compactness result holds for many wavelets $\phi$. In particular, all Schwarz functions and the Haar function.

\begin{theorem}
  Let $\phi \in \cala_1$ and $\sigma \in L^\infty((0,\infty) \times \R^d)$ be thin. Then,
  \[ T_\sigma f(x) = \int \sigma(a,b) W_\phi f(a,b) \phi_{a,b}(x) \, \dfrac{da}{a} db \]
  is compact.
\end{theorem}
The converse also holds if $\phi$ is admissible and Schwartz since $\lip \phi_{a,b},\phi \rip \to 0$ as $d((a,b),(1,0)) \to \infty$, see for example \cite[Appendix, Lemmas 2 and 4]{frazier91}.
This yields the following positivity result, by taking $\Gamma$ in Theorem \ref{thm:compact2} to be the subgroup $ \{1\}\times \R^d$.

\begin{theorem}
  Let $\phi \in \cala_1$ and $\sigma: (0,\infty) \times \R^d \to [0,1]$ be thin. Then there exists $c>0$ such that
  \[ \lip T_{1-\sigma}f,f \rip = \int_{(0,\infty) \times \R^d} (1-\sigma)|W_\phi f|^2 \, d\lambda \ge c \|f\|^2 \]
  for all $f \in L^2(\bbr^d)$.
\end{theorem}

As an immediate consequence, we obtain the following uncertainty principle.

\begin{cor}
Let $\phi \in \cala_1$. If $E \subset \R^d \times (0,\infty)$ is thin, then there exists $c>0$ such that
	\[ \int_{ (0,\infty) \times \R^d \backslash E} |W_\phi f(a,b)|^2 \dfrac{da \, db}{a} \ge c \|f\|^2 \]
for all $f \in L^2(\R^d)$.
\end{cor}

As in the STFT case, this implies $\supp W_\phi f$ can only be thin if $f=0$, and this is sharp in the sense that we cannot improve from vanishing sets to small ones for general $\phi$. If $f$ and $\phi$ are both supported in a ball $B$, then $W_\phi f$ is supported the region $\{ (a,b) : b \in a^{-1}B-B \}$ which contains the strip $(0,\infty) \times B$.

We also mention that these results continue to hold for higher dimensional wavelet transforms such as the shearlet \cite{kutyniok09}, but describing the classes $\BB^1_w$ and $\cala$ is more difficult. However, we mention that Schwarz functions with Fourier support in a bounded set away from the $y$-axis are included in $\BB^1$ without any weight as shown in \cite{kutyniok09}.

\subsection*{Acknowledgements} The authors would like to thank Karlheinz Gr\"ochenig for his comments.

\end{document}